\documentclass[a4paper,11pt]{article}

\def\bbR{\mathbb{R}}

\def\hm{\hphantom{-}}

\def\vextra{\vphantom{\vrule height0.4cm width0.9pt depth0.1cm}}

\usepackage{graphicx}
\usepackage{subfigure}
\usepackage{amsmath,amsfonts,amsthm}
\usepackage{amssymb,latexsym}
\usepackage{fixmath}
\usepackage{mathrsfs,amsbsy}
\usepackage{dsfont}
\usepackage{enumerate}
\usepackage{kbordermatrix}
\usepackage{color}
\usepackage{epstopdf}
\usepackage{multirow}
\usepackage{epstopdf}

\def\bu{\pmb{u}}
\def\bv{\pmb{v}}

\def\bx{\pmb{x}}

\def\bz{\pmb{z}}

\def\NRes{\mbox{NRes}}
\def\QME{\mbox{{\sc qme}}}
\def\DA{\mbox{{\sc da}}}
\def\BL{\mbox{{\sc bl1}}}
\def\BLDU{\mbox{{\sc bl1-du}}}
\def\BLL{\mbox{{\sc bl2}}}
\def\BLLDU{\mbox{{\sc bl2-du1}}}
\def\BLLDUU{\mbox{{\sc bl2-du2}}}

\usepackage{algorithm,algorithmic}

\newtheorem{theorem}{Theorem}[section]
\newtheorem{lemma}{Lemma}[section]

\theoremstyle{definition}

\newtheorem{example}{Example}[section]

\numberwithin{algorithm}{section}
\numberwithin{equation}{section}
\numberwithin{figure}{section}
\numberwithin{table}{section}

\def\sss{\scriptstyle}
\def\scrA{\mathscr{A}}
\def\scrB{\mathscr{B}}

\title{A structure-preserving doubling algorithm for solving a class of quadratic matrix equation with $M$-matrix}

\usepackage[marginal]{footmisc}
\usepackage{authblk}

\author[a]{Cairong Chen\thanks{Corresponding author. Supported partially by the NSFC grant 11901024. Email address: cairongchen@fjnu.edu.cn.}}

\affil[a]{College of Mathematics and Informatics \& FJKLMAA, Fujian Normal University, Fuzhou 350007, China.}

\begin{document}

\date{\today}

\maketitle

\begin{abstract}
Consider the problem of finding the maximal nonpositive solvent $\Phi$ of the quadratic matrix equation (\QME) $X^2 + BX + C =0$ with $B$ being a nonsingular $M$-matrix and $C$ an $M$-matrix such that $B^{-1}C\ge 0$, and $B - C - I$ a nonsingular $M$-matrix. Such \QME\ arises from an overdamped vibrating system. Recently, Yu et al. ({\em Appl. Math. Comput.}, 218: 3303--3310, 2011) proved that $\rho(\Phi)\le 1$ for this \QME. In this paper, we slightly improve their result and prove $\rho(\Phi)< 1$, which is important for the quadratic convergence of the structure-preserving doubling algorithm. Then, a new globally monotonically and quadratically convergent structure-preserving doubling algorithm to solve the \QME\ is developed. Numerical examples are presented to demonstrate the feasibility and effectiveness of our method.
\end{abstract}

{\small
\medskip
{\em 2000 Mathematics Subject Classification}. 15A24, 65F30, 65H10

\medskip
{\em Key words}.
Quadratic matrix equation, structure-preserving doubling algorithm, $M$-matrix, maximal nonpositive solvent, quadratic convergence
}

\section{Introduction}
In this paper, we consider the problem of finding the maximal nonpositive solvent of the following quadratic matrix equation (\QME)
\begin{equation}\label{eq:QME1}
Q_1(X) \equiv \tilde{A}X^2 + \tilde{B}X + \tilde{C} = 0,
\end{equation}
where
\begin{equation*}
\framebox{
\parbox{11cm}{
$\tilde{A}\in \bbR^{n\times n}$ is a diagonal matrix with positive diagonal elements,\\
$\tilde{B}\in \bbR^{n\times n}$ is a nonsingular $M$-matrix and\\
$\tilde{C}\in \bbR^{n\times n}$ is an $M$-matrix such that $\tilde{B}^{-1}\tilde{C} \ge 0$.
}
}
\end{equation*}
Such \QME\ arises from an overdamped vibrating system \cite{Tisseur2000Backward,Tisseur2001The}. By left multiplying $\tilde{A}^{-1}$ \cite{Yu2011On}, without changing the $M$-matrix structure of it, \QME\ \eqref{eq:QME1} can be reduced to the following form
\begin{equation}\label{eq:QME2}
Q_2(X) \equiv X^2 + B X + C = 0,
\end{equation}
where $B$ is a nonsingular $M$-matrix and $C$ is an $M$-matrix such that $B^{-1}C \ge 0$.
It is known that \eqref{eq:QME2} has a maximal nonpositive solvent $\Phi$ under the condition that \cite{Yu2011On}
\begin{equation}\label{eq:cond}
\mbox{$B - C - I$ is a nonsingular $M$-matrix}.
\end{equation}
This solvent $\Phi$ is the one of interest.

Various iterative methods have been developed to obtain the maximal nonpositive solvent of \QME\ \eqref{eq:QME2} with assumption \eqref{eq:cond},  including the Newton's method and Bernoulli-like methods (fixed-point iterative methods) \cite{Yu2011On}, modified Bernoulli-like methods with diagonal update skill~\cite{Kim2016Diagonal}. Newton's method is not competitive in terms of CPU time since there is a generalized Sylvester matrix equation to solve in each Newton's iterative step. The fixed-point iterative methods are usually linearly or sublinearly convergent and sometimes can be very slow \cite{Yu2011On}.

There are many researches on iterative methods for other \QME s; see \cite{Bai2005On,Davis1981Numerical,Guo2003On,Guo2009Detecting,Guo2005Algorithms,
He2002A,Higham2000Solving,Higham2001Solving,Kratz1987Numer,Lu2016Numerical,Meini1998Solving,Yu2010A,Yu2012Conver} and the references therein. Our work here is mainly inspired by recent study on highly accurate structure-preserving doubling algorithm for quadratic matrix equation from
quasi-birth-and-death process~\cite{Chen2019Highly}. Structure-preserving doubling algorithms are very efficient iterative methods for solving nonlinear matrix equations; for more details, the reader is referred to \cite{Chiang2009Convergence,Chu2005A,Chu2004Structure,
Guo2007On,Guo2006A,Huang2015A,Huang2018A,Wang2012Alter} and the references therein.

Yu et al. in \cite{Yu2011On} proved $\rho(\Phi)\le 1$ under \eqref{eq:cond}. In this paper, we will slightly improve their result and prove that $\rho(\Phi)< 1$ under the same condition. This is important, because it is desired for the quadratic convergence of structure-preserving doubling algorithms. Based on our new result about $\rho(\Phi)$, furthermore, we extend the structure-preserving doubling algorithm for (SF1) \cite{Huang2018A} to solve \QME\ \eqref{eq:QME2} and  give the quadratically convergent results.

The rest of this paper is organized as follows. In section~\ref{sec:prelim} we give some notations and state a few basic results on
nonnegative and $M$-matrices. The main results of this paper are presented in section~\ref{sec:MNSol}. Numerical examples are given in section~\ref{sec:egs} to demonstrate the performance of our method. Finally, conclusions are made in section~\ref{sec:concl}.

\section{Notations and preliminaries}\label{sec:prelim}
In this section, we first introduce some necessary notations and terminologies for this paper. $\bbR^{m\times m}$ is the set of all $m\times m$ real matrices,
$\bbR^n=\bbR^{n\times 1}$, and
$\bbR=\bbR^1$. $I_n$ (or simply $I$ if its dimension
is clear from the context) is the $n\times n$ identity matrix.
For $X\in\bbR^{m\times n}$, $X_{(i,j)}$ refers to its $(i,j)$th entry.
Inequality $X\le Y$ means $X_{(i,j)}\le Y_{(i,j)}$ for all $(i,j)$,
and similarly for $X<Y$, $X\ge Y$, and $X>Y$. In particular, $X\ge
0$ means that $X$ is entrywise nonnegative and it is called a nonnegative matrix. $X$ is entrywise nonpositive if $-X$ is entrywise nonnegative.
A matrix $A\in\bbR^{m\times n}$ is {\em positive}, denoted by $A>0$, if all its entries are positive. The same understanding goes to vectors. For a square matrix $X$,  denote by $\rho(X)$ its spectral radius. A matrix $A\in\bbR^{n\times n}$ is called a {\em $Z$-matrix\/} if $A_{(i,j)}\le 0$ for all $i\ne j$.
Any $Z$-matrix $A$ can be written as
$sI-N$ with $N\ge 0$, and it is called an
{\em $M$-matrix\/} if $s\ge \rho(N)$. Specifically, it is a {\em singular $M$-matrix\/} if $s=\rho(N)$, and a {\em nonsingular
$M$-matrix\/} if $s>\rho(N)$.

The following results on nonnegative matrices and $M$-matrices can be found in, e.g., \cite{Berman1994Nonnegative,Meyer2000Matrix}.

\begin{theorem} \label{thm:P-F-theorem}
Let $A\in \bbR^{n\times n}$ be a nonnegative matrix. Then the spectral radius, $\rho(A)$, is an eigenvalue of $A$ and there exist a nonnegative right eigenvector $\bx$ associated with the eigenvalue $\rho(A)$: $A\bx=\rho(A)\bx$.
\end{theorem}

\begin{theorem}\label{thm:MMtx-equiv}
Let $A\in \bbR^{n\times n}$ be a $Z$-matrix. Then the following statements are equivalent:
\begin{enumerate}
  \item [{\rm (a)}] $A$ is a nonsingular $M$-matrix;

  \item [{\rm (b)}] $A^{-1}\ge 0$;

  \item [{\rm (c)}] $A\bu>0$ holds for some positive vector $\bu\in \bbR^n$.

\end{enumerate}
\end{theorem}

\section{The main results}\label{sec:MNSol}
In this section, we give the main results of this paper. The Lemma~\ref{lm:QBD-rho-2} below can be found in~\cite[Theorem~3.1]{Yu2011On}. The first goal of this paper is to further prove $\rho(\Phi)< 1$.
\begin{lemma}\label{lm:QBD-rho-2}
Suppose \eqref{eq:cond}, then \QME\ \eqref{eq:QME2} has a maximal nonpositive solvent $\Phi$ with $\rho(\Phi)\le 1$, also $B+\Phi$ and $B+\Phi-C$ are both nonsingular $M$-matrices.
\end{lemma}

Since $B - C - I$ is a nonsingular $M$-matrix,
by Theorems~\ref{thm:MMtx-equiv}, there exists a positive vector $\bu>0$ in $\bbR^n$ such that
\begin{equation*}
\bv=(B - C - I)\bu >0.
\end{equation*}
Throughout this paper, $\bv$ and $\bu$ are reserved for the ones here. The following lemma is inspired by \cite[Lemma~3.2]{Chen2019Highly}, we still give the proof for completeness.

\begin{lemma}\label{lm:QBD-rho-1}
Suppose \eqref{eq:cond}, i.e.,  $B - C - I$ is a nonsingular $M$-matrix.
Then $\rho(X)\neq 1$ for any nonpositive solvent $X$ of \eqref{eq:QME2}.
\end{lemma}

\begin{proof}
Suppose, to the contrary, that $\rho(X)=1$ (which is equivalent to $\rho(-X)=1$), where $X$ is nonpositive solvent of \eqref{eq:QME2}.
Then according to Theorem~\ref{thm:P-F-theorem},
there exists a nonzero and nonnegative vector $\bz\in \bbR^n$ such that $-X\bz=\bz$ and thus
$$
(X^2+BX+C)\bz=0
$$
to give $(B-C-I)\bz=0$, which contradicts with that $B-C-I$ is nonsingular.
\end{proof}

Combining Lemma~\ref{lm:QBD-rho-2} and Lemma~\ref{lm:QBD-rho-1}, we immediately finish our first goal of this paper. Moreover, we have the following theorem. The theorem is implied by \cite[Theorem 3.1]{Chen2019Highly} or \cite[Theorem 2.3]{Meng2018Condition}.

\begin{theorem}\label{thm:QBD-exist}
Under the assumption \eqref{eq:cond}, the \QME\ \eqref{eq:QME2} has a unique maximal nonpositive solvent $\Phi$.
Moreover, it hold that $ \Phi\le X_0$ and
\begin{equation*}
-\Phi \bu\le \bu-B^{-1}\bv,
\end{equation*}
where $X_0=-B^{-1}C$ is as defined in \eqref{eq:scrA0}.
\end{theorem}

It can be checked that Theorem~\ref{thm:QBD-exist} is applicable to
\begin{equation*}
CY^2 + BY + I = 0,
\end{equation*}
which is called {\em dual equation\/} of \eqref{eq:QME2}. In conclusion, Theorem~\ref{thm:sol-prop} below gives
some of the important results, the proof is similar to that of
\cite[Theorem~3.2]{Chen2019Highly} and thus it is omitted here.

\begin{theorem}\label{thm:sol-prop}
Suppose \eqref{eq:cond}. The following statements hold.
\begin{enumerate}
  \item[{\rm (a)}] We have

        \begin{align*}
         \Phi \le X_0=-B^{-1}C\le 0, \quad -\Phi \bu&\le \bu-B^{-1}\bv, \\
         \Psi\le Y_0=-B^{-1}\le 0, \quad -\Psi \bu&\le \bu-B^{-1}\bv.
        \end{align*}

  \item[{\rm (b)}] $\rho(\Phi)<1$ and $\rho(\Psi)<1$.

  \item[{\rm (c)}] $I-\Phi\Psi$ and $I-\Psi\Phi$ are nonsingular $M$-matrices.
\end{enumerate}
\end{theorem}

Now we are in position to develop a structure-preserving doubling algorithm for solving the \QME\ \eqref{eq:QME2}. Similar to the discussion in the introduction of \cite{Chen2019Highly}, \QME\ \eqref{eq:QME2} is connected with the matrix pencil
\begin{equation}\label{eq:QBD2SF1}
\scrA_0\begin{bmatrix}
                I\\
                X
              \end{bmatrix}
              =\scrB_0\begin{bmatrix}
                                  I\\
                                  X
                            \end{bmatrix}X,
\end{equation}
where
\begin{subequations}\label{eq:scrA-scrB0}
\begin{align}
\scrA_0&=\begin{bmatrix}
                          -B^{-1}C & 0\\
                         \hm B^{-1}C & I
                      \end{bmatrix}=:\kbordermatrix{ &\sss n & \sss n\cr
                                                    \sss n & \hm E_0 & 0 \cr
                                                    \sss n & -X_0 & I}, \label{eq:scrA0}\\
\scrB_0&=\begin{bmatrix}
                         I & \hm B^{-1}\\
                         0 & - B^{-1}
                      \end{bmatrix}=:\kbordermatrix{ &\sss n & \sss n\cr
                                                    \sss n & I & -Y_0 \cr
                                                    \sss n & 0 & \hm F_0} . \label{eq:scrB0}
\end{align}
\end{subequations}
Now that the matrix pencil $\scrA_0-\lambda\scrB_0$ is in (SF1), it is natural for us to
apply the following doubling algorithm (see Algorithm~\ref{alg:DA-SF1}) for (SF1) \cite{Huang2018A} to solve \eqref{eq:QBD2SF1}.

\begin{algorithm} 
\caption{Doubling Algorithm for (SF1) \cite{Huang2018A}}
\label{alg:DA-SF1}
\begin{algorithmic}[1]
    \REQUIRE $X_0,\,Y_0,\, E_0,\,F_0\in\bbR^{n\times n}$ determined by \eqref{eq:scrA-scrB0}.
    \ENSURE  $X_{\infty}$ as the limit of $X_i$ if it converges.
    \hrule\vspace{1ex}
    \FOR{$i=0, 1, \ldots,$ until convergence}
         \STATE compute $E_{i+1},\, F_{i+1},\, X_{i+1},\, Y_{i+1}$ according to
                \begin{align*}
                    E_{i+1}&= E_i(I_n-Y_iX_i)^{-1}E_i, \\
                    F_{i+1}&= F_i(I_n-X_iY_i)^{-1}F_i, \\
                    X_{i+1}&= X_i+F_i(I_n-X_iY_i)^{-1}X_iE_i, \\
                    Y_{i+1}&= Y_i+E_i(I_n-Y_iX_i)^{-1}Y_iF_i.
                \end{align*}

    \ENDFOR
    \RETURN $X_i$ at convergence as the computed solution.
\end{algorithmic}
\end{algorithm}

Theorem~\ref{thm:cvg-QBD-reg} below is essentially \cite[Theorem~6.1]{Chen2019Highly} or \cite[Theorem~4.1]{Guo2006A}. The only difference lies in the initial matrices $(E_0, F_0, X_0,Y_0)$.

\begin{theorem}\label{thm:cvg-QBD-reg}
Under \eqref{eq:cond}, the matrix sequences $\{E_k\}, \{F_k\}, \{X_k\}$
and $\{Y_k\}$ generated by {\em Algorithm~\ref{alg:DA-SF1}} are well-defined and, moreover, for $k\ge 1$,
\begin{enumerate}
  \item [{\rm (a)}] $E_k=(I-Y_k\Phi)\Phi^{2^k}\ge 0$;

  \item [{\rm (b)}] $F_k=(I-X_k\Psi)\Psi^{2^k}\ge 0$;

  \item [{\rm (c)}] $I-X_kY_k$ and $I-Y_kX_k$ are nonsingular $M$-matrices;

  \item [{\rm (d)}] $\Phi\le X_{k}\le X_{k-1}\le 0,\;\Psi\le Y_k\le Y_{k-1}\le 0$, and
          \begin{equation}\label{eq:cvg-bd-reg}
          0\le X_k-\Phi\le \Psi^{2^k}(-\Phi)\Phi^{2^k}, \,\,
          0\le Y_k-\Psi\le \Phi^{2^k}(-\Psi)\Psi^{2^k}.
          \end{equation}
\end{enumerate}
\end{theorem}
From \eqref{eq:cvg-bd-reg} and Theorem~\ref{thm:sol-prop}(b), we can conclude that $X_k$ and $Y_k$ generated by Algorithm \ref{alg:DA-SF1} converge quadratically to $\Phi$ and $\Psi$, respectively,
under \eqref{eq:cond}.

\section{Numerical Examples}\label{sec:egs}
In this section, we will present numerical results obtained with Algorithm~\ref{alg:DA-SF1} for solving \QME\ \eqref{eq:QME2}. We will compare Algorithm~\ref{alg:DA-SF1}(referred to as \DA) with two Bernoulli-like methods presented in \cite{Yu2011On}(referred to, respectively, as \BL\ and \BLL\ as in \cite{Kim2016Diagonal}) and three modified Bernoulli-like methods with diagonal update skill \cite{Kim2016Diagonal}(referred to as \BLDU, \BLLDU\ and \BLLDUU, respectively). In reporting numerical results, we will record the numbers of iterations (denoted by ``Iter''), the elapsed CPU time in seconds (denoted as ``CPU'') and plot iterative history curves for normalized residual \NRes\ defined by
\begin{equation*}
\NRes(X_k)
  =\frac{\|X_k^2+BX_k+C\|_\infty}
        {\|X_k\|_\infty(\|X_k\|_\infty+\|B\|_\infty)+\|C\|_\infty}.
\end{equation*}
All runs terminate if the current iteration satisfies either $\NRes<10^{-12}$ or the number of the prescribed iteration $k_{max}=1000$ is exceeded. All computations are done in MATLAB.

\begin{example}[{\cite{Kim2016Diagonal}}]\label{eg:QME1}
Consider the equation \eqref{eq:QME2} with
$$
B=
    \begin{bmatrix}
      20 &-10 & & & \\
      -10&30  & -10& & \\
      &-10&30&-10&\\
      & &\ddots&\ddots&\ddots\\
      &&&-10&30&-10\\
      &&&&-10&20
    \end{bmatrix},\quad
C=
    \begin{bmatrix}
      15 &-5 & & & \\
      -5&15  & -5& & \\
      &-5&15&-5&\\
      & &\ddots&\ddots&\ddots\\
      &&&-5&15&-5\\
      &&&&-5&15
    \end{bmatrix}.
$$

\setlength{\tabcolsep}{8.0pt}
\begin{table}[!h]
\centering
\caption{Numerical results for Example \ref{eg:QME1}}\label{table:QME1}
\begin{tabular}{|l|ccc|ccc|}\hline
 & \multicolumn{3}{|c|}{$n=30$} & \multicolumn{3}{|c|}{$n=100$} \\ \cline{2-7}
$Method$\hm & Iter & CPU &\NRes & Iter &CPU& \NRes\\ \hline\hline

$\DA$ & $4$ & $0.0009$ & $8.9890\times 10^{-17}$&$4$ & $0.0067$ & $1.0356\times 10^{-16}$ \vextra\\

$\BL$ & $11$ & $0.0010$ & $1.3381\times 10^{-13}$&$11$ & $0.0065$ & $1.3380\times 10^{-13}$\\

$\BLDU$ &$8$ & $0.0008$ & $5.5076\times 10^{-13}$&$8$ & $0.0042$ & $5.5072\times 10^{-13}$\\

$\BLL$ & $13$ & $0.0009$ & $8.4734\times 10^{-13}$&$13$ & $0.0046$ & $8.4734\times 10^{-13}$\\

$\BLLDU$ & $12$ & $0.0016$ & $1.5404\times 10^{-13}$&$12$ & $0.0050$ & $1.5410\times 10^{-13}$\\

$\BLLDUU$ & $10$ & $0.0011$ & $1.0202\times 10^{-13}$&$10$ & $0.0047$ & $1.0198\times 10^{-13}$\\ \hline
\end{tabular}
\end{table}

In Table~\ref{table:QME1}, we record the numerical results for Example \ref{eg:QME1}. We find that \DA\ uses the smallest iteration numbers and delivers the lowest value of \NRes\ within all the tested methods. For this example, in some situations, \DA\ is not the fastest one in terms of elapsed CPU time. The reason is that it needs more cost at each iterative step than other methods and its iteration number is not less enough than other's.
Figure~\ref{Fig:Curves-for-QME1} plots the convergent history for Example \ref{eg:QME1}. Quadratic monotonic convergence of \DA\ and monotonic linear convergence of Bernoulli-like methods clearly show. \qquad$\Diamond$

\begin{figure}[t]
{\centering
\begin{tabular}{ccc}
\hspace{-0.5 cm}
\resizebox*{0.50\textwidth}{0.26\textheight}{\includegraphics{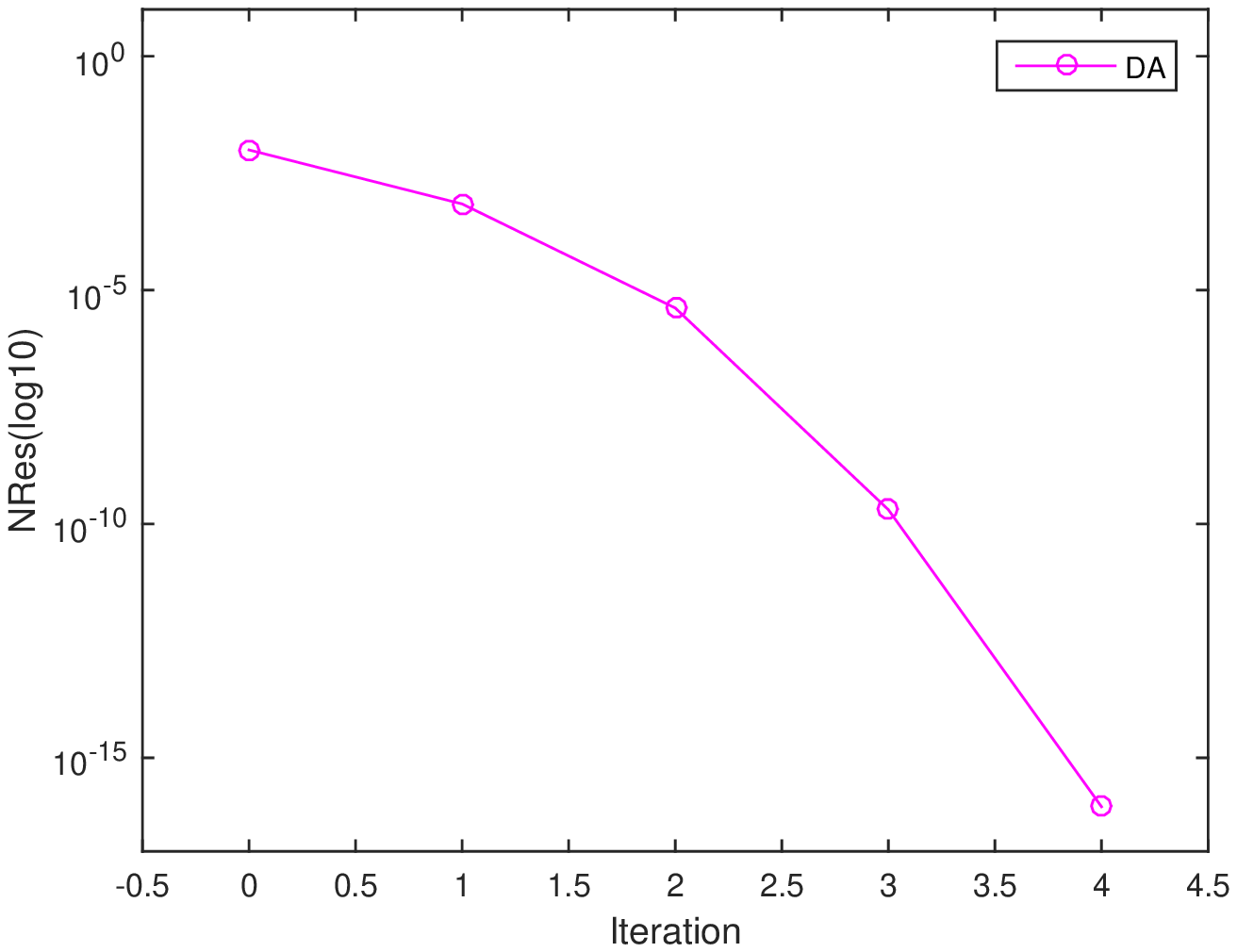}}
& & \hspace{-1 cm}
\resizebox*{0.50\textwidth}{0.26\textheight}{\includegraphics{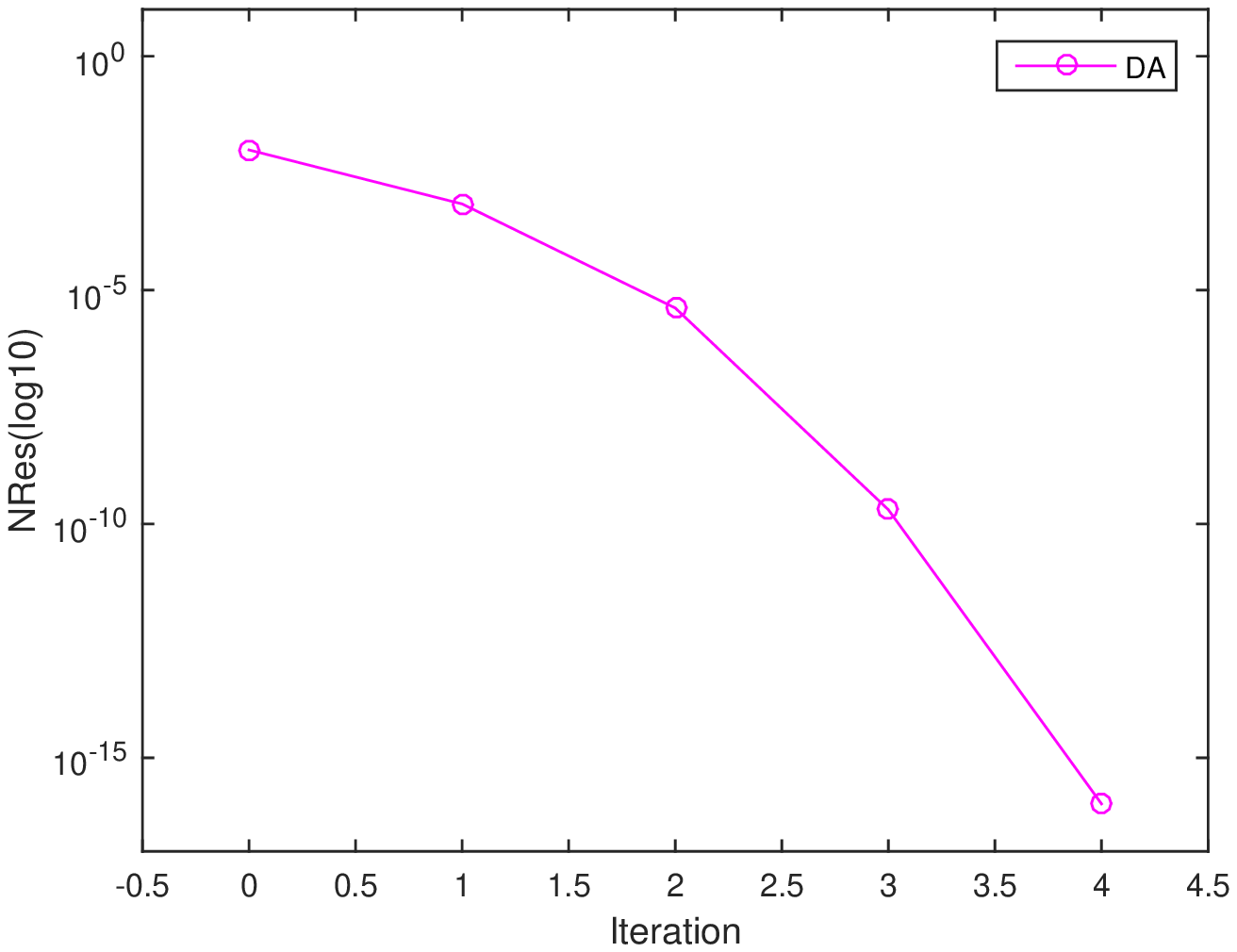}} \vspace{2ex}\\
\hspace{-0.5 cm}
\resizebox*{0.50\textwidth}{0.26\textheight}{\includegraphics{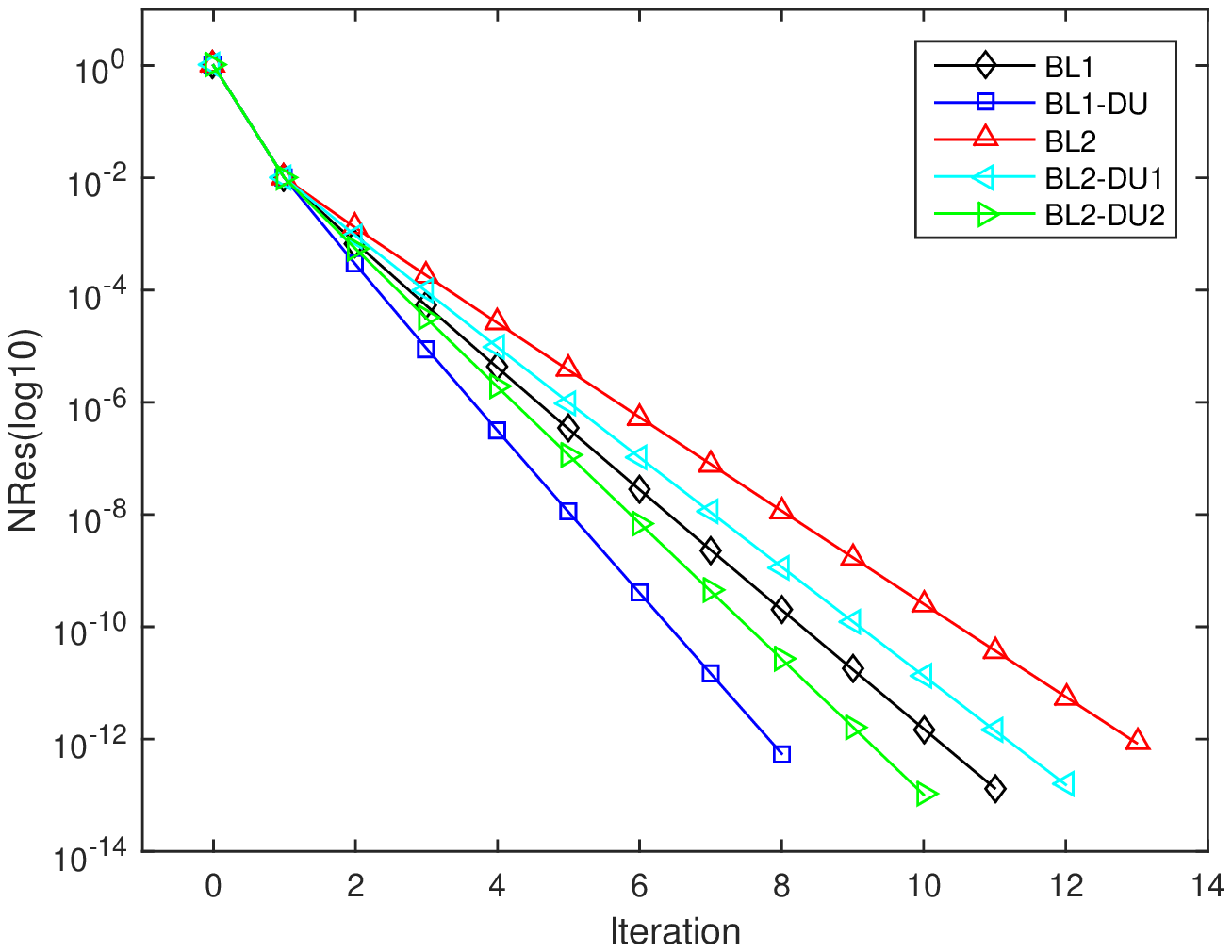}}
& & \hspace{-1 cm}
\resizebox*{0.50\textwidth}{0.26\textheight}{\includegraphics{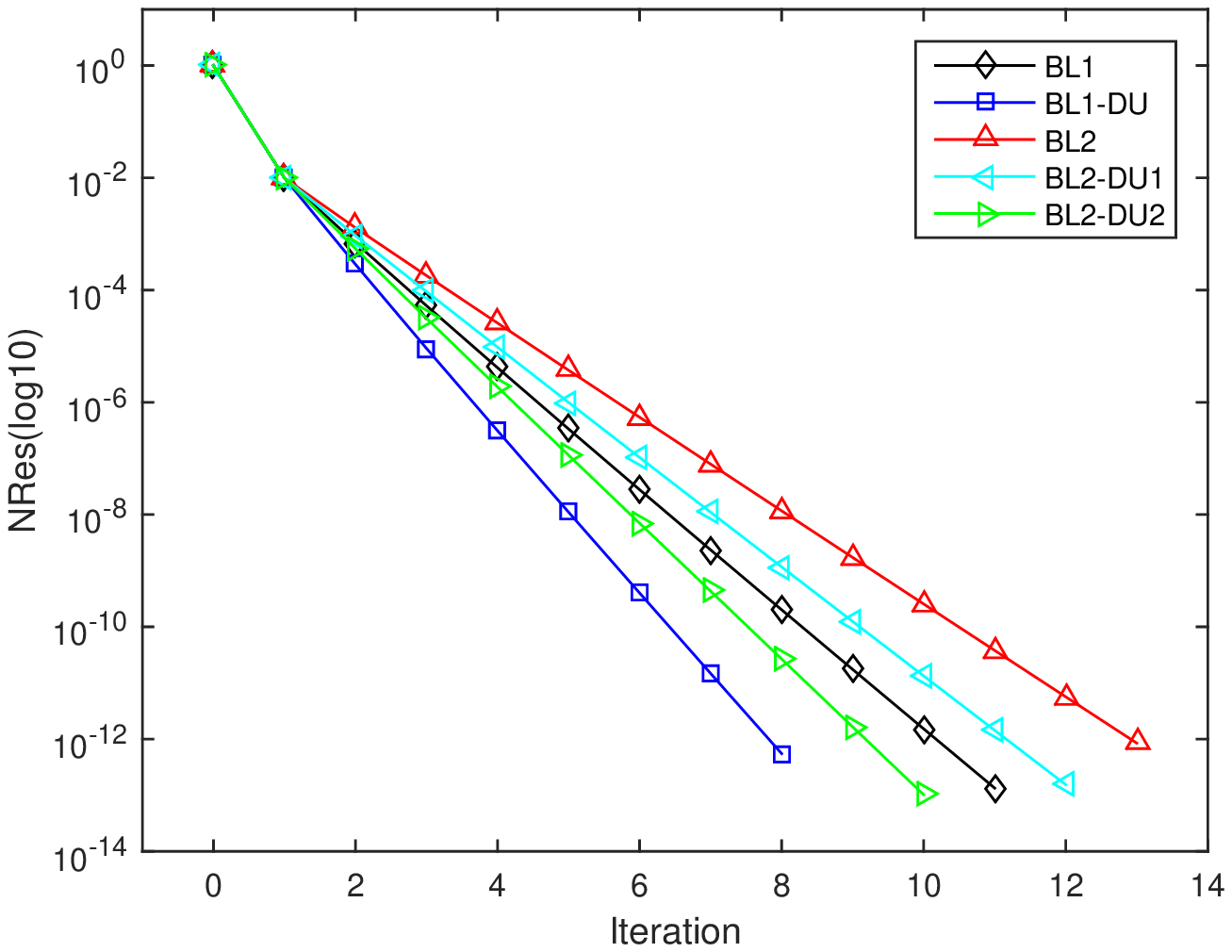}}
\end{tabular}\par
}\vspace{-0.5 cm}
\caption{Convergent history curves for Example \ref{eg:QME1}. The left two are for $n=30$ and the right two are for $n=100$.}
\label{Fig:Curves-for-QME1}
\end{figure}

\end{example}
\begin{example}[{\cite{Kim2016Diagonal}}]\label{eg:QME2}
Consider the equation \eqref{eq:QME2} with
$$
B=
    \begin{bmatrix}
      4 &-1 & & & \\
      -1&4  & -1& & \\
      &-1&4&-1&\\
      & &\ddots&\ddots&\ddots\\
      &&&-1&4&-1\\
      &&&&-1&4
    \end{bmatrix},\quad C=I.
$$

Table \ref{table:QME2} displays the the numerical results for Example \ref{eg:QME2}. We find that \DA\ is the best one for this example in terms of Iter, CPU and \NRes. Figure \ref{Fig:Curves-for-QME2} shows the convergent history for Example \ref{eg:QME2}. Quadratic monotonic convergence of \DA\ and monotonic linear convergence of Bernoulli-like methods again clearly show. \qquad$\Diamond$

\setlength{\tabcolsep}{8.0pt}
\begin{table}[!h]
\centering
\caption{Numerical results for Example \ref{eg:QME2}}\label{table:QME2}
\begin{tabular}{|l|ccc|ccc|}\hline
 & \multicolumn{3}{|c|}{$n=20$} & \multicolumn{3}{|c|}{$n=100$} \\ \cline{2-7}
$Method$\hm & Iter & CPU &\NRes & Iter &CPU& \NRes\\ \hline\hline
$\DA$ & $7$ & $0.0006$ & $1.0236\times 10^{-16}$&$9$ & $0.1212$ & $1.4387\times 10^{-16}$ \vextra\\

$\BL$ & $78$ & $0.0023$ & $8.9153\times 10^{-13}$&$325$ & $0.1620$ & $9.8009\times 10^{-13}$\\

$\BLDU$ &$55$ & $0.0020$ & $7.5483\times 10^{-13}$&$227$ & $0.1306$ & $9.8569\times 10^{-13}$\\

$\BLL$ & $143$ & $0.0038$ & $9.5330\times 10^{-13}$&$637$ & $0.2219$ & $9.7345\times 10^{-13}$\\

$\BLLDU$ & $120$ & $0.0033$ & $9.2154\times 10^{-13}$&$539$ & $0.2355$ & $9.7406\times 10^{-13}$\\

$\BLLDUU$ & $97$ & $0.0027$ & $8.6896\times 10^{-13}$&$441$ & $0.1890$ & $9.7441\times 10^{-13}$\\ \hline
\end{tabular}
\end{table}

\begin{figure}[t]
{\centering
\begin{tabular}{ccc}
\hspace{-0.5 cm}
\resizebox*{0.50\textwidth}{0.26\textheight}{\includegraphics{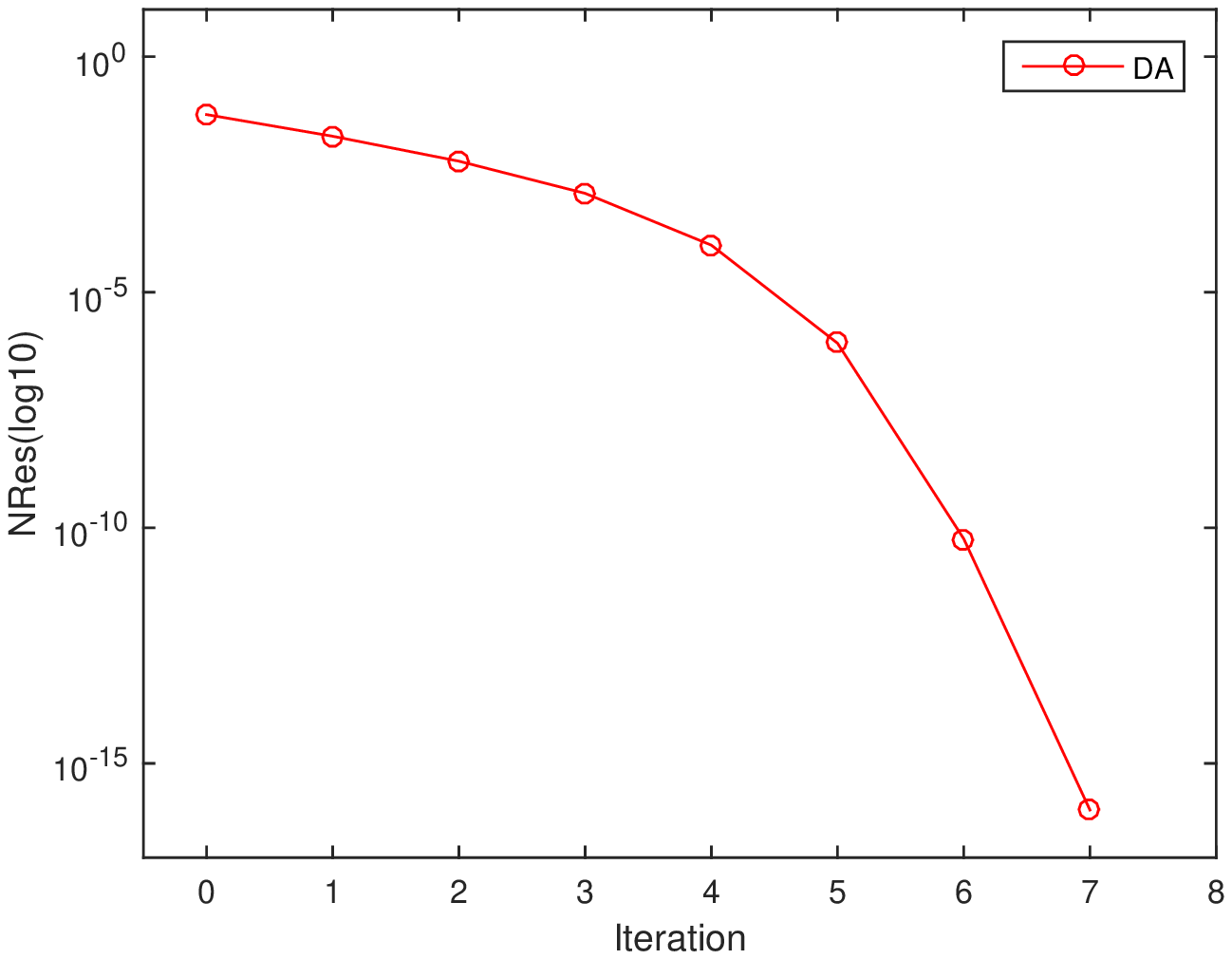}}
& & \hspace{-1 cm}
\resizebox*{0.50\textwidth}{0.26\textheight}{\includegraphics{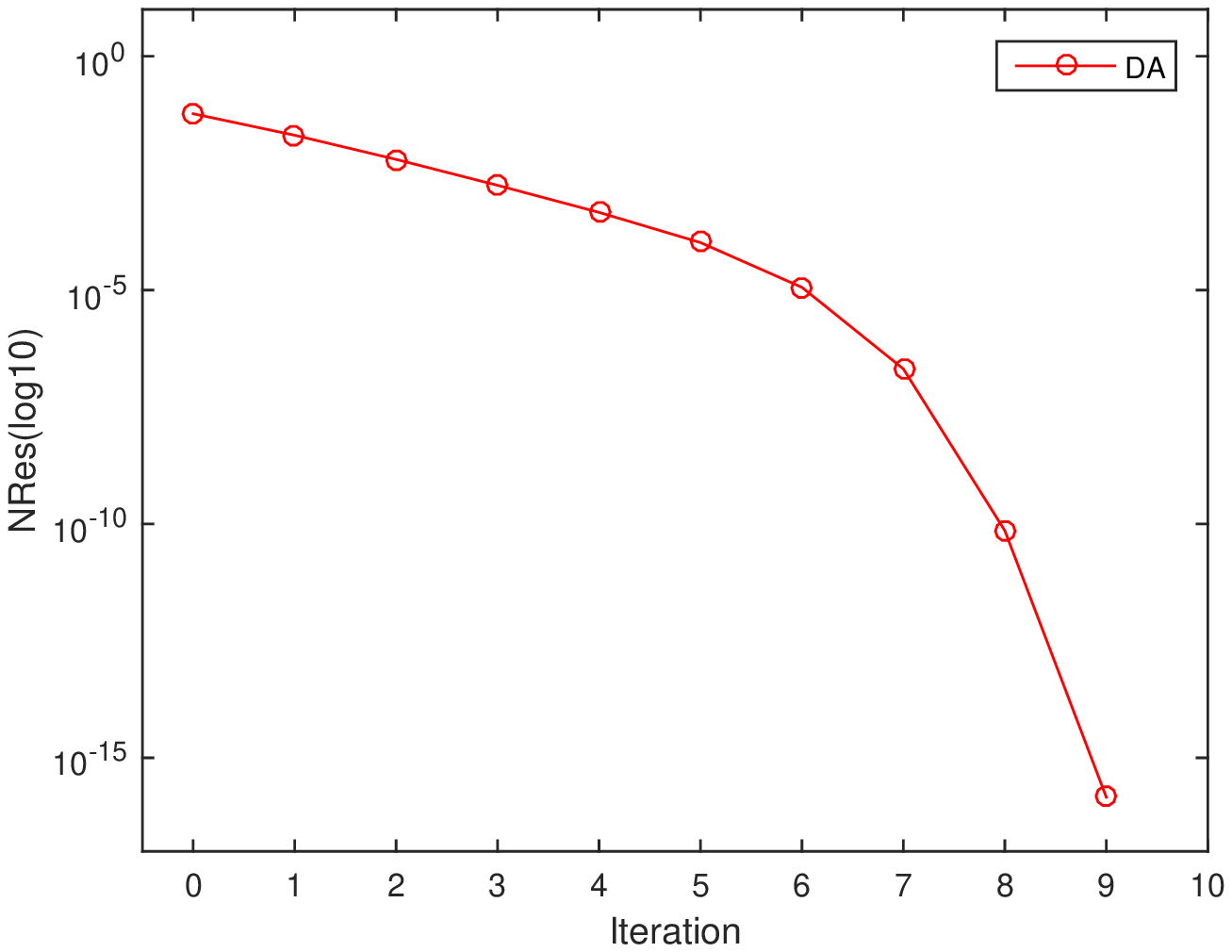}} \vspace{2ex}\\
\hspace{-0.5 cm}
\resizebox*{0.50\textwidth}{0.26\textheight}{\includegraphics{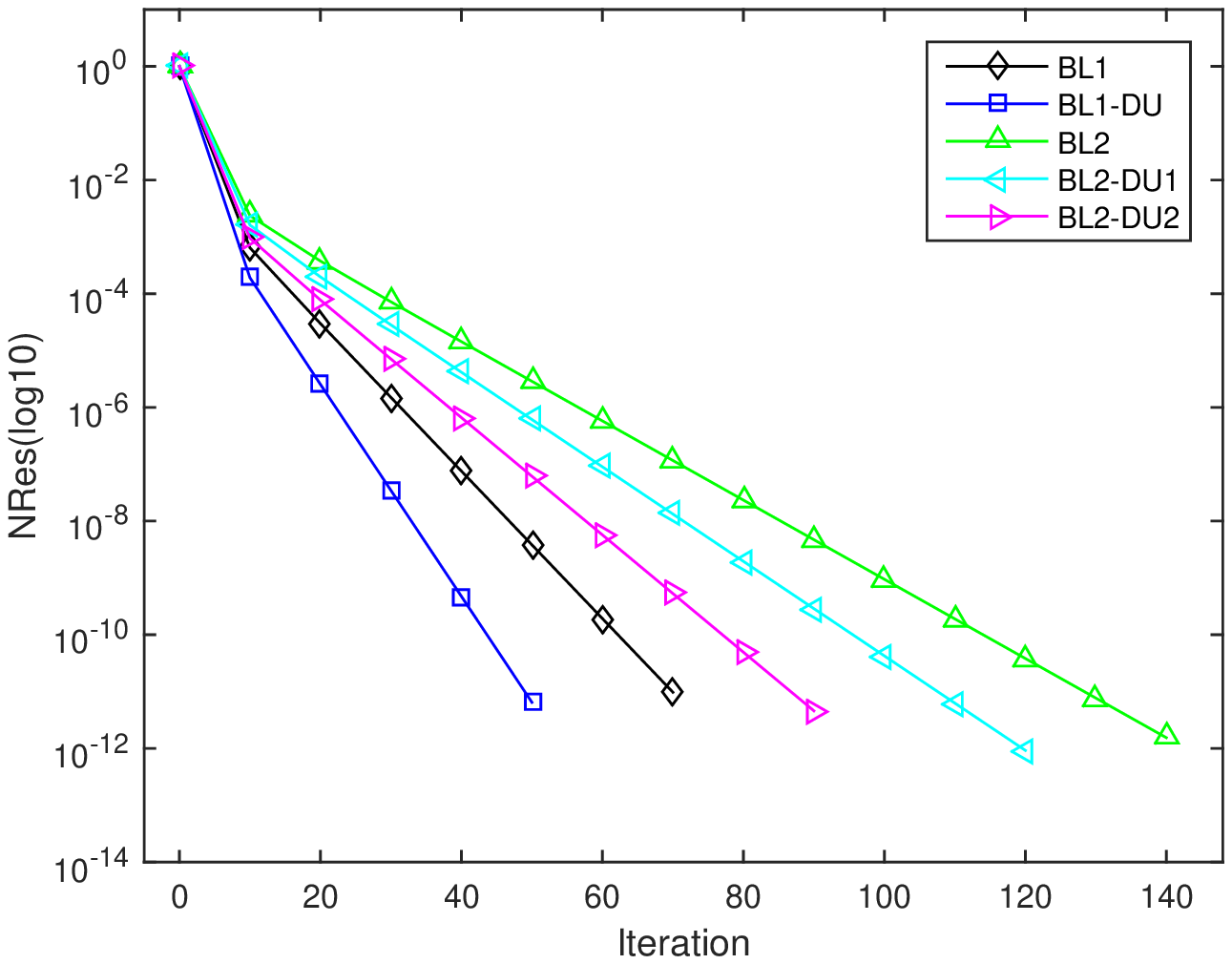}}
& & \hspace{-1 cm}
\resizebox*{0.50\textwidth}{0.26\textheight}{\includegraphics{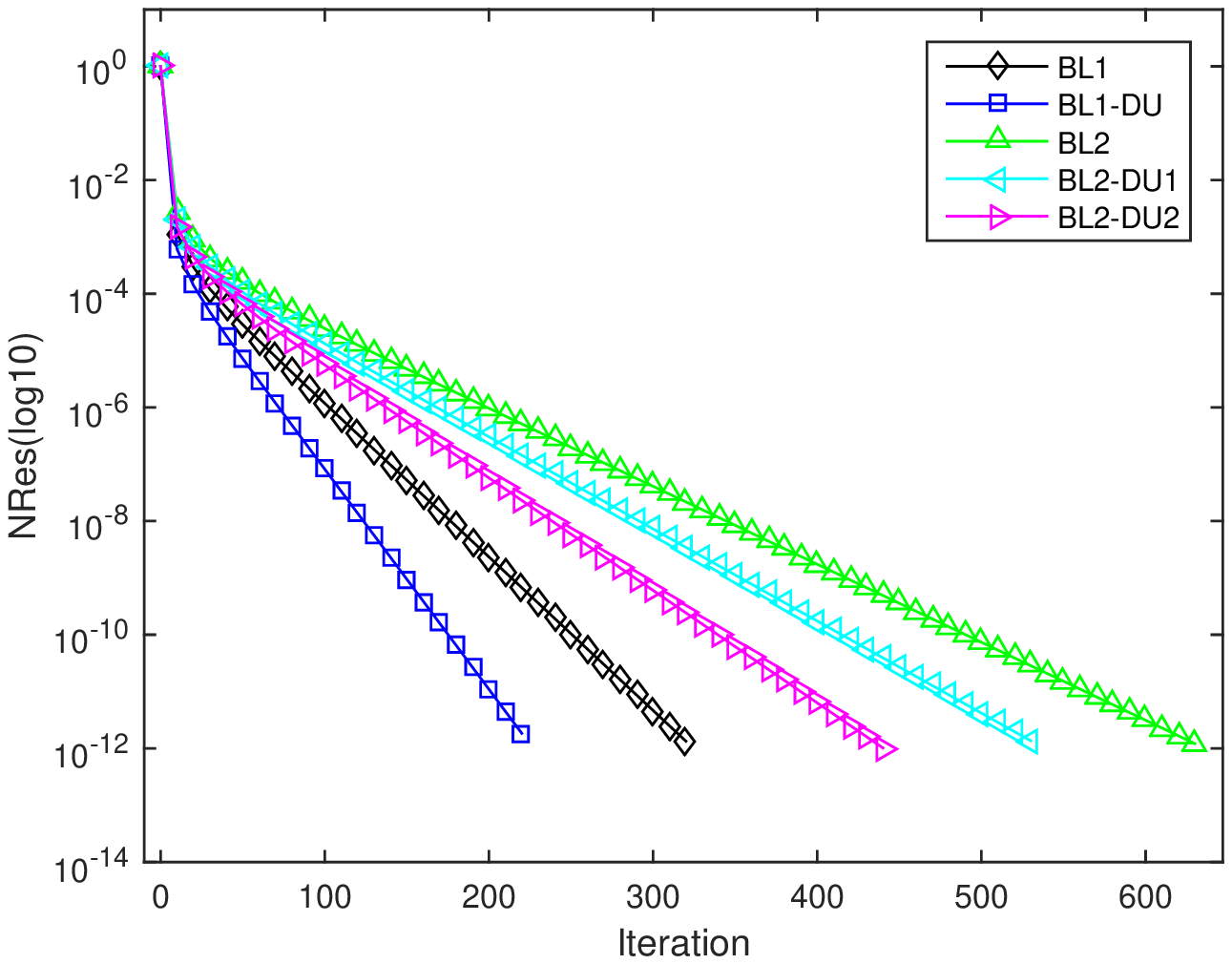}}
\end{tabular}\par
}\vspace{-0.5 cm}
\caption{Convergent history curves for Example \ref{eg:QME2}. The left two are for $n=20$ and the right two are for $n=100$.}
\label{Fig:Curves-for-QME2}
\end{figure}
\end{example}

\section{Conclusions}\label{sec:concl}
The structure-preserving doubling algorithm for (SF1) \cite{Huang2018A} is extended to compute the maximal nonpositive solvent of a type of \QME s. It is shown the approximations generated by the algorithm are globally monotonically and quadratically convergent. Two numerical examples are presented to demonstrate the feasibility and effectiveness of our method. Our work here can be seen as a new application of the structure-preserving doubling algorithm for (SF1).




\end{document}